\documentclass{amsart}
\usepackage{amssymb,
amsmath,
amsthm,
graphicx,
amscd,
amsfonts,
hyperref,
}
\theoremstyle{plain}
\newtheorem{thm}{Theorem}
\newtheorem{lm}[thm]{Lemma}
\newtheorem{cor}[thm]{Corollary}
\newtheorem{prop}[thm]{Proposition}
\theoremstyle{definition}
\newtheorem{de}[thm]{Definition}
\newtheorem{ex}[thm]{Example}
\newtheorem{re}[thm]{Remark}

\newcommand{\RR}{{\mathbb R}}
\newcommand{\ZZ}{{\mathbb Z}}

{\begin{figure} \begin{center}}%
{\end{center} \end{figure}}

\newcommand{\Mg}{\mathcal{M}^\mathrm{trop}_g}

\begin{document}

\begin{abstract}
We prove that in the moduli space of genus-$g$ metric graphs the
locus of graphs with gonality at most $d$ has the classical dimension
\[ \min\{3g-3,2g+2d-5\}. \] 
This follows from a careful parameter count to establish the upper
bound and a construction of sufficiently many graphs with gonality at
most $d$ to establish the lower bound. Here, gonality is the
minimal degree of a non-degenerate harmonic map to a tree that satisfies the
Riemann-Hurwitz condition everywhere. Along the way,
we establish a convenient combinatorial datum capturing such harmonic
maps to trees.
\end{abstract}

\title{On metric graphs with prescribed gonality}

\author[F.~Cools]{Filip Cools}
\address[Filip Cools]{
	Department of Mathematics\\
	Katholieke Universiteit Leuven\\
 	Celestijnenlaan 200b - box 2400, 3001 Leuven, Belgium}
\email{filip.cools@wis.kuleuven.be}

\author[J.~Draisma]{Jan Draisma}
\address[Jan Draisma]{
Department of Mathematics and Computer Science\\
Technische Universiteit Eindhoven\\
P.O. Box 513, 5600 MB Eindhoven, The Netherlands;\\
and Vrije Universiteit Amsterdam, The Netherlands}
\email{j.draisma@tue.nl}

\maketitle
\renewcommand{\phi}{\varphi}

\section{Definitions and result}

\subsection*{Metric graphs}
A {\em topological graph} is a topological space obtained by gluing a
finite, disjoint union of closed intervals along an equivalence relation
on the boundary points. If a topological graph $\Gamma$ is connected and
the intervals from which it is glued are prescribed with a positive length,
then $\Gamma$ becomes a compact metric space with shortest-path metric.
Such a metric space is called a {\em metric graph}. The {\em genus}
of a metric graph is its cycle space dimension. Here is a
metric graph of genus $2$ with edge lengths $a,b,c$:
\begin{center}
\includegraphics{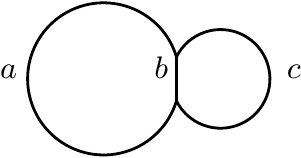}
\end{center}
Every point $v$ in a metric graph has a neighbourhood isometric, for
some positive $\epsilon$, to a finite union of half-open intervals
$[0,\epsilon)$ glued (only) along $0$.  Each of these intervals is called a {\em
half-edge} emanating from $v$, and their number is the {\em valency} of
$v$. We identify two half-edges emanating from $v$ (for different $\epsilon$) if one is contained
in the other.

\subsection*{Harmonic maps}
A map $\phi$ from a metric graph $\Gamma$ to a metric graph $\Sigma$
is called {\em harmonic} if it is continuous, linear with integral
slopes outside a finite number of points, and if it moreover satisfies
the following harmonicity condition at each point $v \in \Gamma$: Fix
a half-edge $e$ emanating from $w:=\phi(v)$, and consider the sum of all
slopes of $\phi$ along half-edges emanating from $v$ that map to $e$. That
sum, denoted $m_\phi(v)$, should be independent of the choice of $e$. 
A harmonic map has a well-defined degree, defined as $\deg \phi:=\sum_{v
\in \Gamma, \phi(v)=w} m_\phi(v)$ for any $w \in \Sigma$.
On the left is an example with $m_\phi(v)=3$, and on the 
right is an example of a degree-2 harmonic map from our earlier
genus-$2$ graph to a metric tree:
\begin{center}
\hfill
\includegraphics{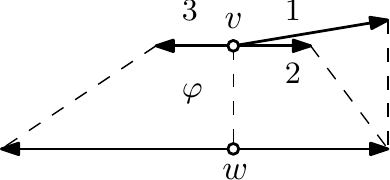}
\hfill 
\includegraphics{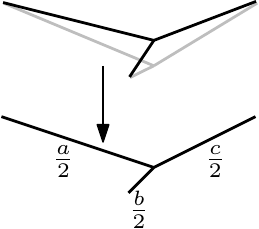}
\hfill 
\ {}
\end{center}

This definition is closely related to the notion of {\em pseudo-harmonic
indexed morphisms} in \cite{Caporaso12} (which generalise harmonic
morphisms from \cite{Baker09}) as follows: there the set-up concerns
ordinary (non-metric) graphs, and the role of the index there is played
by the slope in our definition. Moreover, in \cite{Caporaso12} the
definition is extended to graphs in which the vertices are decorated
with a non-negative genus; we do not do so here.

\subsection*{Riemann-Hur\-witz conditions and tropical morphisms}
Let $\phi:\Gamma\to \Sigma$ be a harmonic map and $v \in \Gamma$. The 
{\em Riemann-Hur\-witz condition} \cite[Definition2.2]{Bertrand11} on $\phi$ at $v$ is the 
inequality \[ k-2 \geq m_\phi(v) \cdot (l-2), \] where $k$ is the valency of $v$ and $l$ the valency of $w:=\phi(v)$.
We will only be interested in harmonic maps that satisfy the
Riemann-Hurwitz condition at all $v \in \Gamma$ and that, moreover, have
{\em non-zero} integral slopes outside a finite number of points. This
latter condition is equivalent to saying that $m_{\phi}(v)$ is strictly
positive at all $v \in \Gamma$. This is related to the condition of {\em
non-degeneracy} in \cite{Caporaso12} as follows: there the non-degeneracy
is required at vertices, and not at ``internal points'' of edges. But
by folding each edge contracted to a point into two, one can create
a harmonic morphism between corresponding metric graphs with non-zero
slopes, albeit that the target graph needs to be modified as discussed
below. The non-zero slope condition is also related to the modifiability
condition in \cite[Remark 2.5]{Bertrand14} and the finiteness condition 
in \cite[Definition 2.4]{ABBR14a}.

We will call harmonic maps satisfying the Riemann-Hurwitz condition
and the non-zero slope condition everywhere {\em tropical morphisms}
between metric graphs.

\subsection*{Modifications}
A {\em modification} of a metric graph $\Gamma$ is any metric graph
$\Gamma'$ obtained from $\Gamma$ by grafting a finite number of
metric trees onto points of $\Gamma$. Given a modification $\Gamma'$
of $\Gamma$ and a tropical morphism $\phi: \Gamma \to \Sigma$, there
exist a modification $\Gamma''$ of $\Gamma'$, a modification $\Sigma'$
of $\Sigma$, and a tropical morphism map $\phi':\Gamma'' \to \Sigma'$
of the same degree as $\phi$ and extending $\phi$, constructed as
follows. For any tree $S$ grafted onto $v \in \Gamma$ when going from
$\Gamma$ to $\Gamma'$, the graph $\Gamma''$ has $m_\phi(v)-1$ additional copies
of $S$ grafted onto $v$ and $m_\phi(v')$ copies of $S$ grafted onto $v'$ for each 
$v'\in\phi^{-1}(\phi(v))\setminus\{v\}$, and $\Sigma'$ has a single copy of $S$ grafted
onto $\phi(v)$. Here we use that $m_\phi(v)$ is positive, which follows
from the fact that slopes are non-zero. The map $\phi'$ equals $\phi$
on $\Gamma$ and maps the copies of $S$ in $\Gamma''$ to the
single copy in $\Sigma'$.

On the right is a modification of the harmonic
morphism on the left: 
\begin{center}
\includegraphics{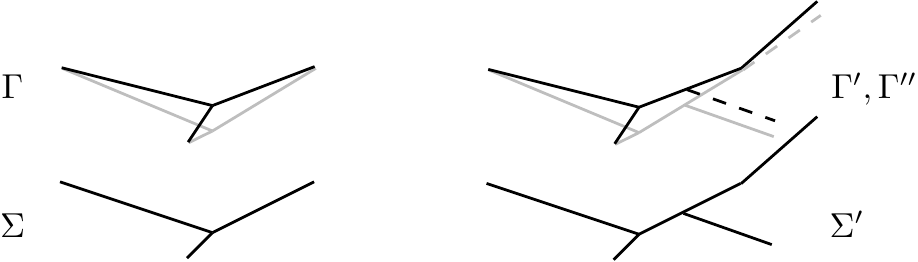}
\end{center}
Here the two solid segments are grafted onto $\Gamma$ to arrive
at $\Gamma'$, and to arrive at $\Gamma''$ the two additional dashed lines
are grafted.

\subsection*{Gonality}
The {\em tree gonality} of a metric graph $\Gamma$ is the minimum degree
of any tropical morphism from any modification $\Gamma'$ of $\Gamma$
to a tree.

There are two notions of graph gonality in the literature, which are both
inspired by the gonality of an algebraic curve. They are tree (or {\em
geometric}) gonality and {\em divisorial} gonality. Each of these comes
in several flavours, e.g. for ordinary or metric graphs \cite{Baker08},
for graphs where the vertices can be decorated with higher genera
\cite{Caporaso12}, and for metrized complexes \cite{Amini15,Luo14}. Yet
another variant is {\em stable gonality}, which is the infimum of
the divisorial gonality over all subdivisions of an ordinary graph
\cite{Cornelissen15}.  In this paper, {\em gonality} will always refer
to tree gonality of metric graphs as defined above.

Given such a tropical morphism $\phi:\Gamma' \to T$ of degree $d$,
and any $w \in T$, the divisor $D':=\sum_{v \in \phi^{-1}(w)} m_\phi(v)
\cdot (v)$ on $\Gamma'$ has degree $d$ and 
rank \cite{Baker06} at least one. Moving the chips of $D$ on grafted
trees to their grafting points, one obtains a divisor $D$ on $\Gamma$
of degree $d$ and rank at least one. In particular, tree gonality is at
least divisorial gonality. Divisorial gonality, in turn, is bounded from
below by the treewidth of a graph or metric graph \cite{Dobben14,Amini14}.

\subsection*{The moduli space of metric graphs}

Fix a natural number $g \geq 2$. For any genus-$g$ graph $G=(V,E)$ (that
is, a connected graph with $|E|-|V|+1=g$; multiple edges and loops are
allowed) with all valencies greater than or equal to $3$, we have $|E| \leq
3g-3$, with equality if and only if all valencies are $3$. The positive
orthant $M_G=\RR_{>0}^E$ of edge lengths parameterises metric graphs of
{\em combinatorial type} $G$. If $H$ is obtained from $G$ by contracting
an edge $e$ (not a loop), then $H$ is another genus-$g$ graph, and we glue
$M_H$ to $M_G$ as the coordinate hyperplane $\RR_{>0}^{E \setminus \{e\}}
\times \{0\}$. Similarly, if $H=(V',E')$ is a graph isomorphic to $G$
via an isomorphism $\sigma:E \to E'$, then we glue $M_H$ to $M_G$ via
the induced bijection $\RR_{>0}^{E'} \to \RR_{>0}^{E}$. In particular,
we do this for automorphisms of $G$.  This gluing of all $M_G$ where $G$
runs over all genus-$g$ graphs yields a topological space known as $\Mg$,
the {\em moduli space of metric graphs of genus $g$}. It has dimension
$3g-3$. By construction, every genus-$g$ graph $G=(V,E)$ defines a map
$\RR_{>0}^E \to \Mg$. Every genus-$g$ metric graph is a modification
of some graph represented by a point in $\Mg$.  The generalisation of
$\Mg$ to graphs with weighted vertices and marked points is discussed,
for instance, in \cite[Section 3]{Caporaso12b}, and the topology of these
spaces is the topic of \cite{Kozlov09}.

\subsection*{Main results}

We will prove the following two theorems.

\begin{thm} \label{thm:dimension}
For $d,g \geq 2$ the locus of metric graphs in $\Mg$ that have gonality at
most $d$ is closed of dimension $\min\{2g+2d-5,3g-3\}$. In particular,
the locus of genus-$g$ metric graphs of gonality at least $\lceil
(g+2)/2 \rceil$ is open and dense in $\Mg$. 
\end{thm}

This theorem comprises two inequalities, and both will be proved by
purely combinatorial means. Using the Kempf-Kleiman-Laksov existence
result for special divisors \cite{Kempf71,Kleiman72} and 
\cite[Corollary 3.25 and Corollary 4.28]{ABBR14a} (which are variants
of Baker's specialisation lemma \cite[Lemma 2.8]{Baker08} and of 
Conrad's result \cite[Corollary B.3]{Baker08}, respectively), 
it follows that in fact {\em all} metric graphs in $\Mg$
have gonality at most $\lceil (g+2)/2 \rceil$. 
We will give a constructive, purely combinatorial
proof for the following statement.

\begin{thm} \label{thm:construction}
Let $g \geq 1$ be a natural number. For any genus-$g$ graph $G=(V,E)$ all
of whose vertices have valency $3$, the positive orthant $(\RR_{>0})^E$
contains a non-empty open cone $C_G$ whose image in $\Mg$ consists
entirely of metric graphs with gonality exactly $d:=\lceil (g+2)/2
\rceil$. Moreover, $C_G$ can be chosen such that every metric graph
represented by a point in $C_G \cap \ZZ_{>0}^E$ has a
degree-$d$ divisor of rank $1$ supported at integral points.
\end{thm} 

It has been conjectured that in fact {\em all} metric graphs with integral
edge lengths admit such a divisor \cite[Conjecture 3.10]{Baker08},
but our methods do not imply that result. We further remark that by
\cite{Hladky13}, on a {\em metric} graph with integral edge lengths,
the rank of a divisor supported on the integral points equals its rank
when regarded as a divisor on the natural {\em ordinary graph} whose
vertices are the integral points (see also \cite{Luo11}). 

In \cite[Theorem B]{Cornelissen15} an upper bound is established for
the {\em stable gonality} of an ordinary graph, which is defined there
as the infimum of the tree gonality over all subdivisions of the graph.
Our Theorem~\ref{thm:construction} implies the special case of that theorem in
which all vertices have valency at most three.

The fact that the locus of metric graphs of gonality (at least)
$\lceil (g+2)/2 \rceil$ is open and dense in $\Mg$ is an exact tropical
analogue of the corresponding statement for algebraic curves.  This is
interesting as, so far, only the (divisorial) gonality of rather specific
graphs (such as chains of loops) was well-understood \cite{Cools10a}.
Using the aforementioned variant of Baker's specialization lemma, our
theorem implies that a general curve of genus $g$ has gonality at least
$\lceil (g+2)/2 \rceil$ (the ``non-existence part'' of Brill-Noether
theory). The idea that one does not need a {\em specific} graph to prove
this statement but that, rather, a suitable dimension count suffices,
goes back to a conference talk by Mikhalkin in 2011 \cite{Mikhalkin11}.

\begin{re}
The dimension $2g+2d-5$ in Theorem~\ref{thm:dimension} is equal to the
dimension of the gonality-$d$ locus in the moduli space of genus-$g$
curves \cite{Segre28}. Applying the tropicalisation map described
in \cite{Abramovich15} to the latter locus might yield a different
proof of the fact that the dimension in our theorem is {\em at least}
$2g+2d-5$: by the specialisation lemma, this map sends the classical
gonality-$g$ locus into our gonality-$g$ locus, and its image should
still have dimension $2g+2d-5$. On the other hand, \cite[Theorem
5.4]{Amini14b} shows an example of a genus-$27$ ordinary metric graph
which, equipped with suitable edge lengths, has gonality $4$ while it is
not the tropicalisation of any $4$-gonal curve of genus $27$.  It seems
very worthwhile to work out the details of the precise correspondence
between the classical gonality-$d$ locus and the corresponding tropical
locus---but our present goal to establish Theorems~\ref{thm:dimension}
and~\ref{thm:construction} using purely combinatorial means.
\end{re}

Our paper is organised as follows. In Section~\ref{sec:Gluing} we
present a combinatorial datum that captures a metric graph together with
a tropical morphism of degree $d$ to a tree. We call this a {\em gluing
datum}. Using this datum, in Section~\ref{sec:Upper} we prove the upper
bound on the dimension of the locus of gonality-$d$ metric graphs in
$\Mg$. Finally, in Section~\ref{sec:Lower} we construct the cone $C_G$
and variants of it for lower-than-maximal gonality, thus showing that
the upper bound is in fact the right dimension.

The most important problem that we leave open is to find a combinatorial
construction of a degree-$\lceil (g+2)/2 \rceil$ tropical morphism $\phi$
from (a modification of) {\em every} metric graph $\Gamma$ of genus $g$
to a tree. Moreover, since our construction of such morphisms for graphs
in the cone $C_G$ depends continuously on the edge lengths, and since the
moduli space of metric graphs of genus $g$ is connected in codimension
one by \cite{Caporaso12b}, it is natural to ask whether the space of
such pairs $(\Gamma,\phi)$ is in a suitable sense connected. Both of
these problems are adressed in forthcoming work by the second author
with Alejandro Vargas \cite{Draisma17b}.

\section{The gluing datum} \label{sec:Gluing}

We define the following combinatorial gadget.

\begin{de}
A {\em gluing datum} $(T,d,\sim)$ is a tuple consisting of a metric
tree $T$, a natural number $d$, and an equivalence relation $\sim$
on the disjoint union $S:=T_1 \sqcup \cdots \sqcup T_d$ of $d$ copies
of $T$ satisfying the following properties, in which $\psi_i:T \to T_i$
stands for the identification of $T$ with its $i$-th copy.
\begin{enumerate}
\item If $v,w \in S$ satisfy $v \sim w$, then there exist $u \in T$ and
$i,j \in [d]$ such that $\psi_i(u)=v \in T_i$ and $\psi_j(u)=w \in T_j$.

\item For any pair $i,j$ the set $\{u \in T \mid \psi_i(u) \sim
\psi_j(u)\}$ has finitely many connected components, each of which
is closed.

\item The topological graph $\Gamma$ obtained from $S$ by
identifying points along the equivalence relation $\sim$ is connected.

\item For every $w \in T$ and $i \in [d]$ we have the inequality 
\[ (\sum_e k_e(w,i)) - 2 \geq m(w,i) (l(w)-2) \]
described further below. 
\end{enumerate}
\end{de}

We will elaborate on the last item further below.  Here is an example with
$d=3$ leading to a genus-3 graph; the colors are purely decorative. In
the last picture, dangling trees have been removed.
\begin{center}
\includegraphics{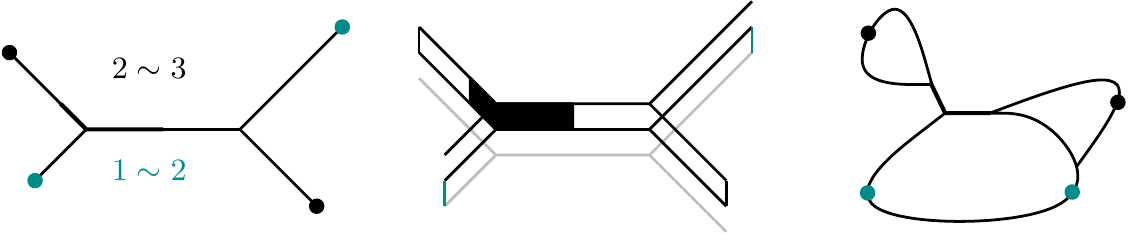}
\end{center}

It is convenient to think of the $d$ copies of $T$ as lying above
each other. Then the first condition says that a point is glued only
to points vertically above or below it. The second condition says that
gluing between two trees happens only along finitely many closed subsets.
Given $u \in T$, if there exist $i,j \in [d]$ for which $u$ is a leaf
of some tree in the forest in (2), then $u$ is called an {\em endpoint}
of the gluing datum.

We turn the topological graph $\Gamma$ into a metric graph as follows: if
for $v \in \Gamma$ the point $\phi(v)$ is not an endpoint, then locally
near $v$ a fixed number $d_v$ of the $T_i$ were glued together. Near
$v$ we give $\Gamma$ the metric of $T$ divided by $d_v$, so that $\phi:
\Gamma \to T$ has slope equal to $d_v$ near $v$.

A gluing datum and a point $w \in T$ together give rise to an equivalence
relation $\sim_w$ on the set $[d]$ defined by $i \sim_w j :\Leftrightarrow
\psi_i(w) \sim \psi_j(w)$.  By the closedness of the gluing point sets,
the map $w \mapsto \sim_w$ is semicontinous in the following sense:
for every $w \in T$ there is an open, connected neigbourhood $U$ of $w$
in $T$ such that $\sim_u$ is constant on each connected component of $U
\setminus \{w\}$ and $\sim_w$ is a coarsening of each of these. 

The last condition in the definition can now be made explicit as follows. Given
$w \in T$ and $i \in [d]$, let $m(w,i)$ be the cardinality of the equivalence
class of $i$ in $\sim_w$ and let $l(w)$ be the valency of $w$ in $T$. For
each half-edge $e \subseteq U$ emanating from $w$ let $k_e(w,i)$ be the
number of equivalence classes into which the $\sim_w$-class of $i$
splits in the refinement $\sim_u$ for $u \in e \setminus \{w\}$. Then
\[ (\sum_e k_e(w,i)) - 2 \geq m(w,i) \cdot (l(w)-2), \]
where the sum is over all half-edges in $T$ emanating from $w$. 

\begin{prop}
Let $(T,d,\sim)$ be a gluing datum, and let $\Gamma=S/\!\!\sim$ be the
topological space obtained from it, equipped with the aforementioned
metric. Then the natural map $\phi:\Gamma \to T$ is a tropical morphism
of degree $d$.
\end{prop}

\begin{proof}
The map $\phi$ is continuous; this follows from the universal property
of the quotient map $S \to \Gamma=S/\sim$, namely, that any map from
$S$ into a topological space such that points equivalent under $\sim$
are mapped to the same point (here the map $S \to T$) factorises through
a continuous map from $\Gamma$ (here $\phi$). The slope of $\phi$
along any half-edge emanating from a point $v$ that is not an endpoint is 
$d_v>0$.

To see that $\phi$ is harmonic, let $v \in \Gamma$ and let $e$ be a half-edge
emanating from $w:=\phi(v)$ in $T$. Then the points near (but unequal to)
$w$ along $e$ induce a fixed equivalence relation $\sim_e$ on $[d]$,
which refines the equivalence relation $\sim_w$ on $[d]$.  Let $I
\subseteq [d]$ be the equivalence class of $\sim_w$ consisting of
those $i$ for which $v$ is the image of $\psi_i(w)$. This class
decomposes as a disjoint union $I_1 \sqcup \cdots \sqcup I_k$ of
equivalence classes under $\sim_e$. The latter classes correspond bijectively
to the half-edges of $\Gamma$ emanating from $v$ that map to $e$. On
the half-edge
corresponding to $I_j$, the metric has been defined such that the slope of
$\phi$ equals $|I_j|$. Adding up all these slopes yields $m_\phi(v):=|I|$,
which is an invariant of $v$ and in particular independent of the
half-edge
$e$. This proves harmonicity. The degree is $d$ because that is the sum
of the cardinalities of all equivalence classes of $\sim_w$.

Finally, to establish the Riemann-Hurwitz conditions for $\phi$ at the
the image $v$ in $\Gamma$ of $\psi_i(w)$, we note that $m(w,i)$ is the
sum of all the slopes of $\phi$ along each edge $e$ emanating from $w$
and that $\sum_e k_e(w,i)$ is the valency of $v$ in $\Gamma$. Hence
the inequality (4) is precisely the Riemann-Hurwitz condition. 
\end{proof}

As we will prove next, every tropical morphism to a tree arises from
a suitable gluing datum. This datum is not unique. First, there is the
obvious ambiguity arising from permuting the copies $T_1,\ldots,T_d$. But
in fact there is more ambiguity, as will become apparent in the following
proof.

\begin{prop} 
Let $\phi:\Gamma \to T$ be a tropical morphism. Then there exists a
gluing datum $(T,d,\sim)$ that gives rise to $\phi$.
\end{prop}

\begin{proof}
The number $d$ is defined as the degree of $\phi$. Without loss
of generality, we may assume that the pre-image of some leaf $w_0$
of $T$ consists of $d$ distinct leaves $v_1,\ldots,v_d$ of $\Gamma$
with $m_{\phi}(v_i)=1$ for all $i$. This situation can be achieved
by grafting an additional interval on some leaf of $T$ and extending the
morphism as discussed in the section on modifications.

We now construct the gluing relation $\sim$ on the disjoint union
$T_1 \sqcup \cdots \sqcup T_d$ of $T$, by determining at points $w \in T$ the local relation
$\sim_w$, or equivalently, the partition $\Pi_w$ of $[d]$ into the equivalence classes of $\sim_w$.  

We will do this by walking along $T$ and starting at $w_0$. 
For the gluing relation $\sim_{w_0}$ at the leaf $w_0\in T$, choose for $\Pi_{w_0}$
the partition $\Pi_1$ of $[d]$ into singletons.
Now we begin to move $w$ along $T$, starting at $w_0$ and tracing the pre-image $\{v_1,\ldots,v_d\}$ under $\phi$ along. 
Let $w'\in T$ be a point above which some of the $v_i$ converge. 
For all points $w$ on $T$ in between $w_0$ and $w'$, we keep the partition $\Pi_w=\Pi_1$. 
However, above $w'\in T$, we are forced to mimic the convergence in the gluing datum by making $\Pi_{w'}$
the corresponding coarsening of $\Pi_1$. 
Let $e_1,\ldots,e_l$ be the half-edges emanating from $w'$ in $T$, where $e_1$ is in the direction of $w_0$. 
Along the remaining half-edges $e_2,\ldots,e_l$ emanating from $w'$ (so all except for $e_1$), we still
have to choose the partitions $\Pi_2,\ldots,\Pi_l$ of $[d]$. They should
satisfy the following rules. First, each $\Pi_i$ should be a refinement
of $\Pi_{w'}$. Second, the number of parts of $\Pi_i$ into which the part
of $\Pi_{w'}$ corresponding to a $v \in \phi^{-1}(w')$ splits equals the
number of half-edges in $\Gamma$ emanating from $v$ that map to $e_i$. Third,
the cardinalities of these parts should equal the slopes of $\phi$ along
those half-edges. As $\phi$ is harmonic, these slopes add up to the same number
$m_\phi(v)$ for all $i$, so that such partitions $\Pi_2,\ldots,\Pi_l$
certainly exist. Now fix any partitions $\Pi_2,\ldots,\Pi_l$ that satisfy the above rules.
Analogously, we can continue moving the point $w\in T$ and define the partition $\Pi_w$ until we have done this for all $w\in T$.
The corresponding gluing datum $(T,d,\sim)$ gives rise to
the map $\phi$ and satisfies the inequality (4) since
$\phi$ satisfies the Riemann-Hurwitz conditions.
\end{proof}

One can make use of the ambiguity in the construction of the gluing datum
$(T,d,\sim)$ (in particular, in the particular choice of the partitions
$\Pi_2,\ldots,\Pi_l$) to impose convenient extra conditions on the gluing
datum, for instance as follows:

\begin{lm} \label{lm:Intervalgluing}
Any tropical morphism $\phi:\Gamma \to T$ arises from a gluing datum
$(T,d,\sim)$ in which for each $i \neq j$ in $[d]$ the set $\{w \in
T \mid i \sim_w j\}$ is homeomorphic to a disjoint union of closed
intervals.
\end{lm}

The closed intervals in the lemma are allowed to consist of a single
point.
  
\begin{re}
Originally, we believed that the condition that any two trees are
glued along intervals also {\em implies} the Riemann-Hurwitz conditions. However,
as pointed out by Alejandro Vargas, this is incorrect. The
following shows a piece of a gluing datum that satisfies the
condition of the lemma but does not satisfy the
Riemann-Hurwitz condition at the central point:
\begin{center}
\includegraphics{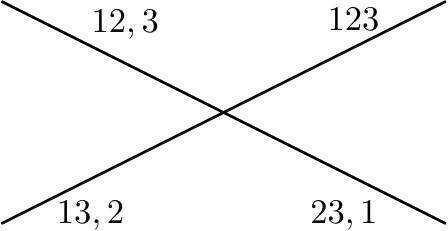}
\end{center}
Here $12,3$ is short-hand for the partition $\{\{1,2\},\{3\}\}$, etc. So
the valency of the metric graph is $k=7$ at this vertex, that of the
tree is $l=4$, $m_\phi$ equals $3$ but $k-2 \not \geq 3 \cdot (l-2)$.
\end{re}

\begin{proof}[Proof of Lemma~\ref{lm:Intervalgluing}]
The condition in the lemma translates to the following combinatorial
condition on the partitions $\Pi_1,\ldots,\Pi_l$ along the half-edges
emanating from a point $w \in T$: if $i,j$ are distinct elements
in the same part of some partition $\Pi_h$, then there is at most
one $h' \neq h$ such that $i,j$ are in the same part of $\Pi_{h'}$.
Lemma~\ref{lm:Partition} below shows that, when $\Pi_1$ is prescribed,
the Riemann-Hurwitz condition implies that $\Pi_2,\ldots,\Pi_l$ can be
chosen to ensure this. 
\end{proof}

\begin{lm} \label{lm:Partition}
Let $l,m,k_1,\ldots,k_l$ be positive integers and assume that
$k_1+\cdots+k_l-2 \geq m(l-2)$. For each $h=1,\ldots,l$ let $\pi_h$
be a partition of the {\em number} $m$ with $k_h$ (nonzero)
parts; in particular, each $k_h$ is at most $m$.

Then there exist partitions $\Pi_1,\ldots,\Pi_l$ of the {\em set} $[m]$
such that $\pi_h$ records the sizes of the parts in $\Pi_h$ and such
that the coarsest common refinement of any three of the $\Pi_h$ is the
partition into singletons.
\end{lm}

We dub this property of a sequence of partitions the {\em triple
intersection property}.

\begin{proof}
We proceed by induction on $m$. The statement is trivially true for $m=1$:
the only choice for each $\Pi_h$ is the partition of $[1]$ into the one
singleton $\{1\}$, and this choice satisfies the triple intersection
property. Now let $m \geq 2$, assume that the statement is true for
$m-1$, and consider $k_1,\ldots,k_l$ and partitions $\pi_h$ of $m$
as in the lemma. Order the $k_h$ such that $k_1 \leq \ldots \leq k_l$.

If $k_2=m$, then the only choices for $\Pi_2,\ldots,\Pi_l$ are the
partitions into singletons, and any choice for $\Pi_1$ will
do. Hence we may assume that $k_1,k_2 < m$, so that $\pi_1,\pi_2$ contain
parts $a_1,a_2>1$, respectively. Next, we have $k_3 > 1$ because otherwise
\[ \sum_h k_h - 2 \leq 3 \cdot 1 + (l-3) \cdot m - 2 < (l-2)m, \] 
where we use that $m>1$. Hence $\pi_3$ contains at least two parts
$a_3,b_3 > 0$. Similarly, we find that  $k_4 > m/2$, hence each $\pi_h$
with $h \geq 4$ contains at least one $1$. 

Now construct $\pi_h'$ from $\pi_h$ as follows: for $h=1,2$ reduce
the part $a_h$ by $1$; for $h=3$ replace the two parts $a_3,b_3$ by a
single part $a_3+b_3-1$; and for $h \geq 4$ discard a part equal to $1$.
This yields $l$ partitions of $m-1$, and in the inequality of the lemma
both sides have been reduced by $l-2$. By the induction hypothesis,
there are partitions $\Pi_h',\ h=1,\ldots,l$ of $[m-1]$ corresponding
to $\pi_{h}'$ that satisfy the triple intersection property.

From the $\Pi_h'$ we construct partitions $\Pi_h$ of $[m]$ as follows. For
$h=1,2$ let $A_h'$ be the (or a) part of $\Pi_h'$ of cardinality $a_h-1$;
add $m$ to this set to obtain $A_h$. For $h>3$ add the singleton $\{m\}$
to the partition $\Pi_h'$. Finally, for $h=3$ let $A_3'$ be the part of
$\Pi_3'$ of size $a_3+b_3-1$. We want to replace $A_3'$ with two sets
$B_3$ and $A_3:=(A_3' \setminus B_3) \cup \{m\}$ where $B_3 \subseteq
A_3'$ has cardinality $b_3$. The only triple intersection that might
now get cardinality $2$ is the one between $A_1,A_2,A_3$ (which contains
$m$), but this happens only if we put an element of $A_1' \cap A_2' \cap
A_3'$ in $A_3$.  Since this intersection contains at most one element,
we can avoid this by putting that element, if it exists, into $B_3$
(whose prescribed cardinality $b_3$ is positive).
\end{proof}

\begin{re}
We think that a generalisation of Lemma~\ref{lm:Partition} might hold,
where one replaces $2$ by an $n \in \{0,\ldots,l-1\}$, the inequality by 
$k_1+\cdots+k_l-n \geq m(l-n)$, and the triple
intersection property by the property that the coarsest common refinement
of any $n+1$ of the partitions be the partition into singletons. But since
we do not need this for our current purposes, we have not pursued this.
\end{re}

\begin{re}
Lemma~\ref{lm:Intervalgluing} will not be used in the remainder of the
paper. It may, however, prove useful when actually generating
many points in the gonality-$d$ locus of $\Mg$.
\end{re}

Now that we know that gluing datums give rise to tropical morphisms and
vice versa, we can express the genus of a metric graph $\Gamma$ in terms
of a gluing datum as follows.

\begin{prop} \label{prop:genus}
Let $(T,d,\sim)$ be a gluing datum. For each subset $I \subseteq [d]$
define
\[ T_I:=\{w \in T \mid \forall i,j \in I: i \sim_w j\} \subseteq T. \]
Then the genus of the metric graph $\Gamma$ determined by the datum equals
\[ g(\Gamma)=\sum_{I \subseteq [d]} (-1)^{|I|} c(T_I), \]
where $c(T_I)$ is the number of connected components of $T_I$.
\end{prop}

In this proposition, $T_I$ is the image in $T$ of the intersection of
the images in $\Gamma$ of the $T_i$ with $i \in I$, and a closed forest
in $T$. Note that $T_I$ is equal to $T$ for $|I|<2$.


\begin{proof}
The genus of $\Gamma$ equals $1$ minus its Euler characteristic. Since the
Euler characteristic of a forest is its number of connected components,
the formula follows immediately from the usual inclusion-exclusion formula
for the Euler characteristic of simplicial complexes---the additional term
$1$ corresponds to the term with $I=\emptyset$ on the right-hand side.
\end{proof}

\section{Upper bounds on the dimension of the locus of bounded gonality}
\label{sec:Upper}

Our goal in this section is to derive the upper bound from
Theorem~\ref{thm:dimension} on the dimension of the locus in $\Mg$ where
the gonality is equal to $d$. For a gluing datum $(T,d,\sim)$ recall
that $w \in T$ is called an endpoint if it is a leaf of some connected
component of $T_{\{i,j\}}$ for some distinct $i,j \in [d]$; and we write
$E$ for the set of endpoints that have valency at most $2$ in $T$. We
also introduce the following notation: if $\Gamma$ is the corresponding
metric graph and $\phi:\Gamma \to T$ the tropical morphism defined by
the datum and $v \in \Gamma$ a point with valency $k$ where $\phi(v)$
has valency $l$ in $T$, then we set $r_\phi(v):=(k-2)-m_\phi(v)(l-2)$. By
the Riemann-Hurwitz condition, this is a nonnegative number. Note that
$r_\phi(v)$ is positive only at a finite number of points.

\begin{prop} \label{prop:ineq}
Let $(T,d,\sim)$ be a gluing datum, $\phi:\Gamma \to T$ the corresponding
tropical morphism, and $g$ the genus of $\Gamma$. Then we have 
\begin{equation} \label{eq:ineq}
|E|+\sum_{v \in \Gamma: \phi(v) \text{ has valency } > 2} r_\phi(v) \leq 2g+2d-2. 
\end{equation}
\end{prop}

\begin{proof}
Since $\Gamma$ is connected by assumption, we may order the copies of
$T$ such that each $T_i$ with $i>1$ is glued to at least one $T_j$ with
$j<i$. For each $e \in [d]$, the restriction of $\sim$ to
$T_1 \sqcup \ldots \sqcup T_e$ then yields a gluing datum $(T,e,\sim)$ of a connected
metric graph $\Gamma_e$, which is obtained from $\Gamma_{e-1}$ by suitably
gluing the copy $T_e$ to it.

We argue by induction on $e$. For $e=1$, the statement is true:
$E=\emptyset$, $r_\phi(v)=0$ for all $v$, and $0+0=2 \cdot 0+2\cdot 1-2$. For
the induction step, assume that the formula holds for $(T,e-1,\sim)$, and consider the graph $\Gamma_e$. 
Let $C\subset \Gamma$ be the set along which $T_e$ is glued to
$\Gamma_{e-1}$. The increase in $g$ when passing from $\Gamma_{e-1}$
to $\Gamma_e$ equals $c(C)-1$, and the increase in $d$ is $1$. Hence
the right-hand side in~\eqref{eq:ineq} increases with $2c(C)$.

On the other hand, consider a connected component $C'$ of $C$ and let
$v \in C'$. Denote by $l$ the valency of $\phi(v)$ in $T$ and by $k_v'$ the
valency of $v$ in the closed tree $C'\subset \Gamma$. Then $m_\phi(v)$ increases by $1$
in passing from $\Gamma_{e-1}$ to $\Gamma_e$, while the valency $k$ of
$v$ in these graphs increases by $l-k_v'$. This means that $r_\phi(v)$
increases by $l-k_v'-(l-2)=2-k_v'$ (so a decrease if $k_v'>2$).

If $C'$ consists of the single point $v$ only, then $v$ either contributes
to an increase of $|E|$ by at most $1$ (if $l\leq 2$)
or to an increase by $2-k_v'=2$ of the second summand in the 
left-hand side of~\eqref{eq:ineq} (if $l>2$). In either case, $C'$
contributes an increase of at most $2$ to the left-hand side.

Now suppose that $C'$ is not a single point. Then each leaf $v$ of $C'$
has $k'_v=1$ and contributes at most $1$ to the left-hand side: to $|E|$
if $l\leq 2$ or to the other summand if $l>2$. A $v \in C'$
which is not a leaf cannot be a new point in $E$. Indeed,
if it is a leaf of a connected
component of $T_{\{i,e\}}$ with $i \neq e$ and also has valency at
most two in $T$, then it has valency exactly two, and any half-edge
emanating from $v$ not contained in $T_{\{i,e\}}$ (of which
there are one or two) must be contained
in some $T_{\{j,e\}}$ with $j \neq i,e$. But then $v$ is a leaf of the
component of $T_{\{i,j\}}$ containing $v$, hence already in $E$ and not
a new endpoint.  Hence $C'$ contributes an increase of the left-hand
side in~\eqref{eq:ineq} by at most
\[ \#\{\text{leaves of $C'$}\} + \sum_{v \text{
non-leaf in } C'} (2-k'_v). \]
Straightforward combinatorics shows that this quantity equals $2$ for
every tree that is not a single point. 

Since each component of $C$ contributes at most $2$, the left-hand side
of \eqref{eq:ineq} increases by at most $2c(C)$ in passing from
$\Gamma_{e-1}$ to $\Gamma_e$, and combining this with the first paragraph
of the proof we find that the inequality is preserved.
\end{proof}

\begin{cor}
In the setting of Proposition~\ref{prop:ineq} we have $|E| \leq 2g+2d-2.$
\end{cor}

\begin{proof}
This follows immediately from the Riemann-Hurwitz condition that
$r_\phi(v) \geq 0$ for each $v$.
\end{proof}

We can now prove that the dimension of the locus in $\Mg$ where the
gonality is at most $d$ is at most $2g+2d-5$.

\begin{proof}[Proof of Theorem~\ref{thm:dimension}, upper bound.]
Start by taking a gluing datum $(T,d,\sim)$ that gives rise to a
degree-$d$ tropical morphism $\phi:\Gamma \to T$, and assume that
$\Gamma$ has genus $g$. If $v$ is a leaf of $T$ such that $\sim_v$ is
the partition of $[d]$ into singletons, then $\phi^{-1}(v)$ consists
of $d$ valency-one points.  If the interval $e$ leading to $v$
contains points where the gluing is not trivial, then let $w \in e$
the point closest to $v$ with this property. Otherwise, let $w \in T$
be the point of valency greater than two where $e$ is attached to the
rest of $T$. By deleting the segment $(w,v]$ we obtain a new gluing
datum $(T',d,\sim)$ defining a metric graph $\Gamma'$ of which $\Gamma$
is a modification. Proceeding in this manner with deleting unnecessary
leaves, we arrive at a gluing datum, which we still denote $(T,d,\sim)$,
such that $\sim_v$ is non-trivial at every leaf $v \in T$. This means
that all leaves are endpoints. Next re-attach an interval of positive
length at each leaf $v$ where the gluing relation is constant and equal
to $\sim_v$. These intervals do not contribute to the edge lengths of
$\Gamma$ accounted for in the moduli space, and we introduce them merely
because they {\em do} contribute to the parameter count that follows.
We denote the resulting gluing datum again by $(T,d,\sim)$.

Now we count the cardinality $|V \cup E|$ where $V \subseteq T$ is the
set of points of valency greater than $2$ and $E$ is the set of endpoints
of valency at most $2$. By basic combinatorics, $|V|$ is at most $l-2$
where $l$ is the number of leaves, with equality if and only if all
points in $V$ are trivalent. Thus, by the corollary above, $|V \cup E|
\leq 2g+2d-4+l$.  Since all leaves of $T$ are end points, the set $V\cup
E$ can be seen as a vertex set of $T$.  Hence, the complement of $V
\cup E$ in $T$ consists of at most $2g+2d-5+l$ intervals whose lengths
determine the lengths in $\Gamma$. Among these intervals, $l$
intervals at the leaves do not contribute to the image of $\Gamma$ in
the moduli space (these intervals may be longer than the ones added in
the last modification, if the gluing relation was already constant on
a postive-length interval leading into a leaf $v$). Thus we arrive at
the correct dimension count $2g+2d-5$.

To describe the locus in
$\Mg$ where the gonality is at most $d$ one now proceeds as follows:
\begin{enumerate}
\item Enumerate all (finite, combinatorial) trees with at most
$2g+2d-5$ edges (whose vertices are allowed to have
any valency greater than or equal to $1$); there are finitely many of these. 
\item Equip each such tree with a gluing relation $\sim$
that is constant along the edges; again, there are finitely many
possibilities. 
\item Select those combinatorial choices of a tree $T$ plus a gluing that
lead to a connected graph $\Gamma$ of genus $g$ and a map $\Gamma \to T$
that satisfies the Riemann-Hurwitz conditions. (In practice, one might
want to exploit Lemma~\ref{lm:Intervalgluing} in the previous step to
rule out some of the possibilities.)
\item Writing $E$ for the edge set of $T$, we obtain a map $\Psi$ from
$\RR_{\geq 0}^E$ into $\bigsqcup_{g' \leq g}
\mathcal{M}_{g'}^{\text{trop}}$ by interpreting the entries of 
$\ell \in \RR_{\geq 0}^E$ as edge lengths on $T$---we take the union over all $g' \leq g$ since cycles get contracted if their total length is zero. The
set $\Psi(\Psi^{-1}(\Mg))$ is a closed cell in $\Mg$ of dimension at
most $2g+2d-5$. 
\end{enumerate}
The union of these finitely many cells is the locus in the theorem. 
\end{proof}

\section{Constructing graphs with prescribed gonality}
\label{sec:Lower}

In this section we construct a subset of $\Mg$ of dimension
$\min\{2g+2d-5,3g-3\}$ where the gonality is at most $d$, thus proving
the lower bound in Theorem~\ref{thm:dimension}. Consequently, we prove
that for $d=\lceil \frac{g}{2} \rceil+1$ this subset intersects each
cell in the definition of $\Mg$ in a non-empty open subset.  The main
construction is given in the following subsection.

\subsection*{Gluing in a tripod}
Let $(T,d,\sim)$ be a gluing datum with corresponding tropical
morphism $\phi:\Gamma \to T$. Pick points $u,v,w \in \Gamma$ with images
$u':=\phi(u),v':=\phi(v),w':=\phi(w) \in T$. These points span a tree in
$T$ with at most three leaves; let $y$ be the unique point in $T$ that
lies on all the shortest paths $u'v',u'w'$, and $v'w'$. Let $a,b,c$
be the lengths of the shortest paths $u'y,v'y,w'y$, respectively,
and pick additional lengths $a',b',c'>0$. Let $i,j,k \in [d]$ be such
that $u,v,w$ lie in $T_i,T_j,T_k$, respectively. Attach intervals of
lengths $a',b',c'$ to $T$ at $u',v',w'$, respectively, and call the resulting
tree $T'$. Let $u'',v'',w'' \in T'$ be the boundary points of those intervals.
Then one obtains a new gluing datum $(T',d+1,\sim')$ from $(T,d,\sim)$ by
gluing $T'_{d+1}$ only along the points $u'',v'',w''$ with $T_i,T_j,T_k$,
respectively, and leaving the gluing relation on the other copies of $T'$
unchanged. Let $\Gamma'$ be the corresponding metric graph. The
following is now straightforward. 

\begin{lm}
The graph $\Gamma'$ is a modification of the graph obtained from $\Gamma$
by adding a new trivalent vertex with new edges attached to $u,v,w$
of lengths $a+2a',b+2b',c+2c'$, respectively.
\end{lm}

As an example, in the figure below, the graph $\Gamma$ together with its tropical morphism $\phi:\Gamma \to T$ is pictured on the left. The right hand side shows the tropical morphism $\phi':\Gamma' \to T'$, where the top tree is attached to $\Gamma$, and next to it, the graph $\Gamma'$ with the dangling trees removed. 
\begin{center}
\includegraphics[scale=.9]{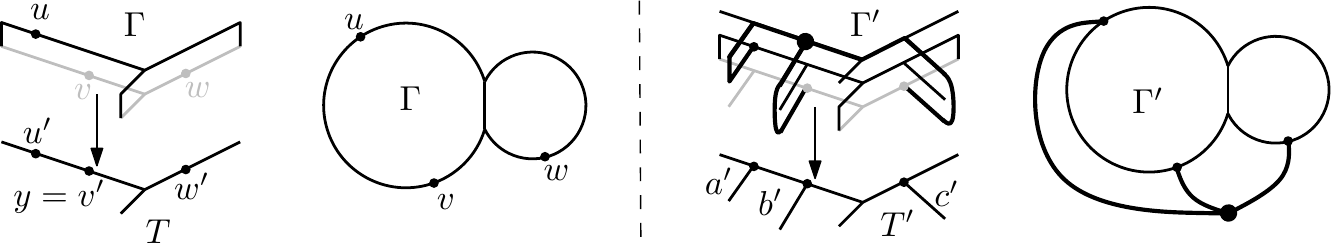}
\end{center}

\subsection*{A lower bound on the dimension in Theorem~\ref{thm:dimension}}

\begin{proof}[Proof of Theorem~\ref{thm:dimension}, lower bound]
We proceed by induction on $d \geq 2$ to exhibit a subset $\mathcal{S}_{g,d}$ of
$\Mg$ of dimension $e(g,d)=\min\{2g+2d-5,\dim \Mg\}$ where the gonality
is at most $d$. The subset will be presented by a finite graph $G=(V,E)$
all of whose vertices are trivalent, together with an open polyhedral
cone $U$ of dimension $e$ in $\RR_{>0}^{E}$ consisting of edge lengths for
which the gonality is at most $d$.

For $d=2$ we are concerned with the locus in $\Mg$ of hyperelliptic
metric graphs. This locus is well-understood \cite{Chan13}, and its dimension
is $2g-1=e(g,2)$ for all $g \geq 1$. Here is a gluing
datum witnessing this dimension:\\

	\begin{center}
  \includegraphics[scale=.8]{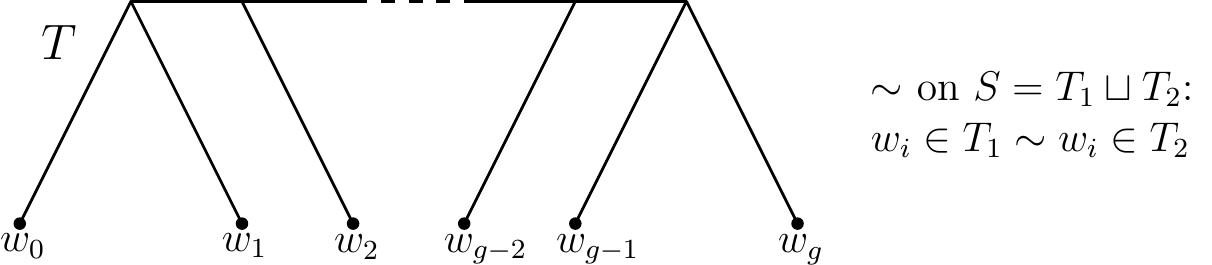}
\end{center}

Next let $d \geq 3$ and assume that we have found suitable subsets
$\mathcal{S}_{g,d-1}$ for all $g$. Then to find $\mathcal{S}_{g,d}$
we pick a suitable subset $\mathcal{S}_{g-2,d-1}$, representated by
$(G,U)$. Note that $e(g,d)=e(g-2,d-1)+6$ provided that $g-2$ is at least
$2$, since then both of the terms in the minimum increase by $6$. If
$g-2 \geq 2$, then any metric graph of genus $g-2$ has only finitely
many automorphisms. Consequently, the number of degrees of additional
freedom when gluing in a tripod is also $6$: this counts the positions
of $u,v,w$ and the positive numbers $a',b',c'$, and the count does not
drop modulo finitely many automorphisms. Hence we are done for $g \geq
4$. 

We deal with the cases $g=2,3$ separately. Every graph with $g=2$ is
hyperelliptic, so for $\mathcal{S}_{g,d}$ with $d \geq 2$ we can take any of the
cells defining $\mathcal{M}^\mathrm{trop}_2$. This leaves the case where $g=3$ and
$d=3$. Here $e(3,3)=6$, and we obtain a subset $\mathcal{S}_{3,3}$ by gluing in
a tripod in a cycle---the reason that this raises the dimension by $5$
rather than $6$ is that the automorphism group of the cycle can move
one of the points, say $u$, to any fixed position.
\end{proof}

\subsection*{Realising all combinatorial types}

We now have almost all ingredients for pro\-ving
Theorem~\ref{thm:construction}, but one more notion is needed for the
existence of a rank-one divisor supported at integral vertices. A finite
subset $S$ of a metric graph $\Gamma$ is called an {\em integral set}
if $\Gamma \setminus S$ is a union of open intervals of length $1$
and half-open intervals of length strictly smaller than $1$. The metric
graph $\Gamma$ has an integral set if and only if it is either a line
segment of arbitrary length, or a single cycle of integral length,
or it has at least one vertex of valency at least three and every line
segment connecting two such vertices has integral length. In the last
case, the integral set is unique, and the closed ends of the half-open
intervals are necessarily valency-one vertices of $\Gamma$.

\begin{proof}[Proof of Theorem~\ref{thm:construction}.]
We proceed by induction on $g$ to construct an open cone $C_G \subseteq
\RR_{>0}^E$ for each trivalent graph $G$ of genus $g$ such that graphs
$\Gamma$ represented by points in $C_G$ have gonality at most $\lceil
(g+2)/2 \rceil$. Moreover, we will do this in such a way that each
metric graph $\Gamma$ corresponding to a point in $\ZZ^E \cap C_G$ has
a modification $\Gamma'$ that admits a tropical morphism $\phi:\Gamma'
\to T$ of degree $\lceil (g+2)/2 \rceil$ to a tree $T$ with the following
additional properties:
\begin{enumerate}
\item $\phi$ maps some integral set $S' \subseteq \Gamma'$ into some integral 
set in $T$; 
\item $S:=S' \cap \Gamma$ is an integral set in $\Gamma$; and 
\item there exists a point $v_0 \in S$ such that $\phi^{-1}(\phi(v_0))
\cap \Gamma \subseteq S$.
\end{enumerate}
Then the divisor $\sum_{v' \in \phi^{-1}(\phi(v_0))} m_\phi(v') v'$
on $\Gamma'$ has rank one, and this remains the case if we move the
chips on $\Gamma' \setminus \Gamma$ to their nearest point on $\Gamma$;
since the points where trees are grafted onto $\Gamma$ to obtain $\Gamma'$
are necessarily in the integral set $S$, this latter divisor is supported
on $S$.

For $g=1$ it is not quite clear how even to define a graph of genus $1$
in which all vertices have valency $3$, but we take as definition the
{\em circle} with no vertices. For this graph the statement is clear:
for any length $a>0$ prescribed to the circle, it has a $2:1$-morphism
$\phi$ to an interval of length $a/2$. If $a$ is integral, then after
choosing an integral set $S$ in the cycle, we can choose $\phi$
such that $\phi^{-1}(\phi(S))=S$.

For $g=2$ there are two possible combinatorial types, and the
statements to be proved are well known for both (the marked points $w$ on $T$
have equivalence class $\{1,2\}$ for $\sim_w$):
	\begin{center}
  \includegraphics[scale=.8]{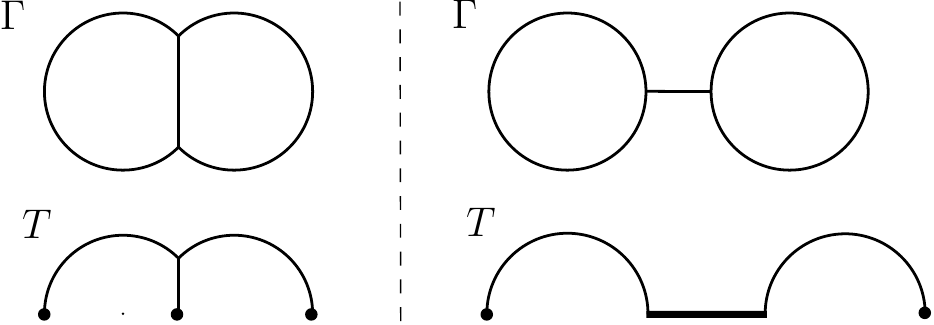}
	\end{center}
In each of the two cases, the open cone equals $\RR^E_{>0}$, and also
the integrality statements are readily verified.

Next, assume that $G=(V,E)$ is a trivalent graph of genus $g>2$. If $G$
has a vertex $y$ that can be removed without disconnecting the graph,
then let $G'=(V',E')$ be the graph obtained by removing $y$ and its
three incident edges $e_1,e_2,e_3$. Note that $G'$ has genus $g-2$. By
assumption, there is an open cone $C_{G'}\subseteq \RR_{>0}^{E'}$ of
dimension $\dim \mathcal{M}^\mathrm{trop}_{g-2}$ consisting of edge
lengths leading to metric graphs with gonality at most $\lceil \frac{g}{2}
\rceil$. By gluing in a tripod we find an open cone $C_G  \subseteq \RR^E$
of the right dimension where the gonality is at most $1+\lceil \frac{g}{2}
\rceil=\lceil \frac{g+2}{2} \rceil$: the inequalities for $C_G$ are
those for $C_{G'}$ plus the conditions that the lengths of $e_1,e_2,e_3$
are sufficiently large (in the terminology of the subsection on gluing
a tripod: larger than $a,b,c$, respectively).

To see that the integrality conditions are preserved, let $\Gamma$ be
the metric graph corresponding to a point in 
$C_G \cap \ZZ^E$. The
restriction to $E'$ defines a metric graph $\Sigma$ of combinatorial type
$G'$, and by the induction hypothesis there is tropical morphism $\psi$
from a modification $\Sigma'$ of $\Sigma$ to a tree $K$ of degree $\lceil
\frac{g}{2} \rceil$ that satisfies the integrality conditions: $\Sigma'$
has an integral set $R'$ such that $\psi(R')$ is contained in an integral
set $U$ of $K$, $R:=\Sigma \cap R'$ is an integral set in $\Sigma$, and
$v_0 \in R$ is such that $\psi^{-1}(\psi(v_0))\cap \Sigma \subseteq R$.

Let $\phi$ be the tropical morphism $\Gamma' \to T$ obtained via gluing
a tripod to the points $u,v,w \in \Sigma$ of $e_1,e_2,e_3$ with
edge lengths $a+2a',b+2b',c+2c'$ as in the subsection on tripods, and
with central vertex $y$. Here $\Gamma'$ is a modification of $\Gamma$.
Then $u,v,w \in R$\footnote{Actually, there is one case where this is
not automatic, namely, when $\Sigma$ is a single cycle. But in this case
we can take $\Sigma'$ equal to $\Sigma$ and {\em choose} $R'$ such that it contains
$u,v,w$.}, and hence $\psi(u),\psi(v),\psi(w) \in U$. We extend $U$
to an integral set $V$ of $T$ by adding the vertices on the new edges
(of lengths $a',b',c' \in \frac{1}{2} \ZZ$) at an integral distance
from $\psi(u),\psi(v),\psi(w)$, respectively. Next, we extend $R'$
to an integral set $S'$ of $\Gamma'$ by
\[ S':=R' \cup \{v \in \Gamma' \setminus \Sigma' \mid \phi(v) \in V\}. \] 
Set $S:=S' \cap \Gamma$. Then we find that each $x \in
\phi^{-1}(\phi(v_0)) \cap \Gamma$ is either in $\psi^{-1}(\psi(v_0))
\cap \Sigma \subseteq R \subseteq S$, or else in $\Gamma' \setminus
\Sigma'$ and hence, since $\phi(x)=\phi(v_0) \in U$, also $x \in S$. Thus $S'$
has the required property.

This concludes the proof for the case where $G$ has a trivalent
vertex $y$ such that removing $y$ does not disconnect the graph. If, on
the other hand, removing {\em any} vertex $y$ of $G$ disconnects $G$,
then any two distinct simple cycles in $G$ intersect in at most one
vertex, i.e., $G$ is a {\em cactus graph}. This case is dealt with by
the proposition below.
\end{proof}

\begin{ex} \label{ex:final}
In the following example, we see two metric graphs of genera $2$ and $4$,
respectively, on the left, along with tropical morphisms of degrees $2$
and $3$, as constructed above: \\
\begin{center}
\includegraphics[scale=.68]{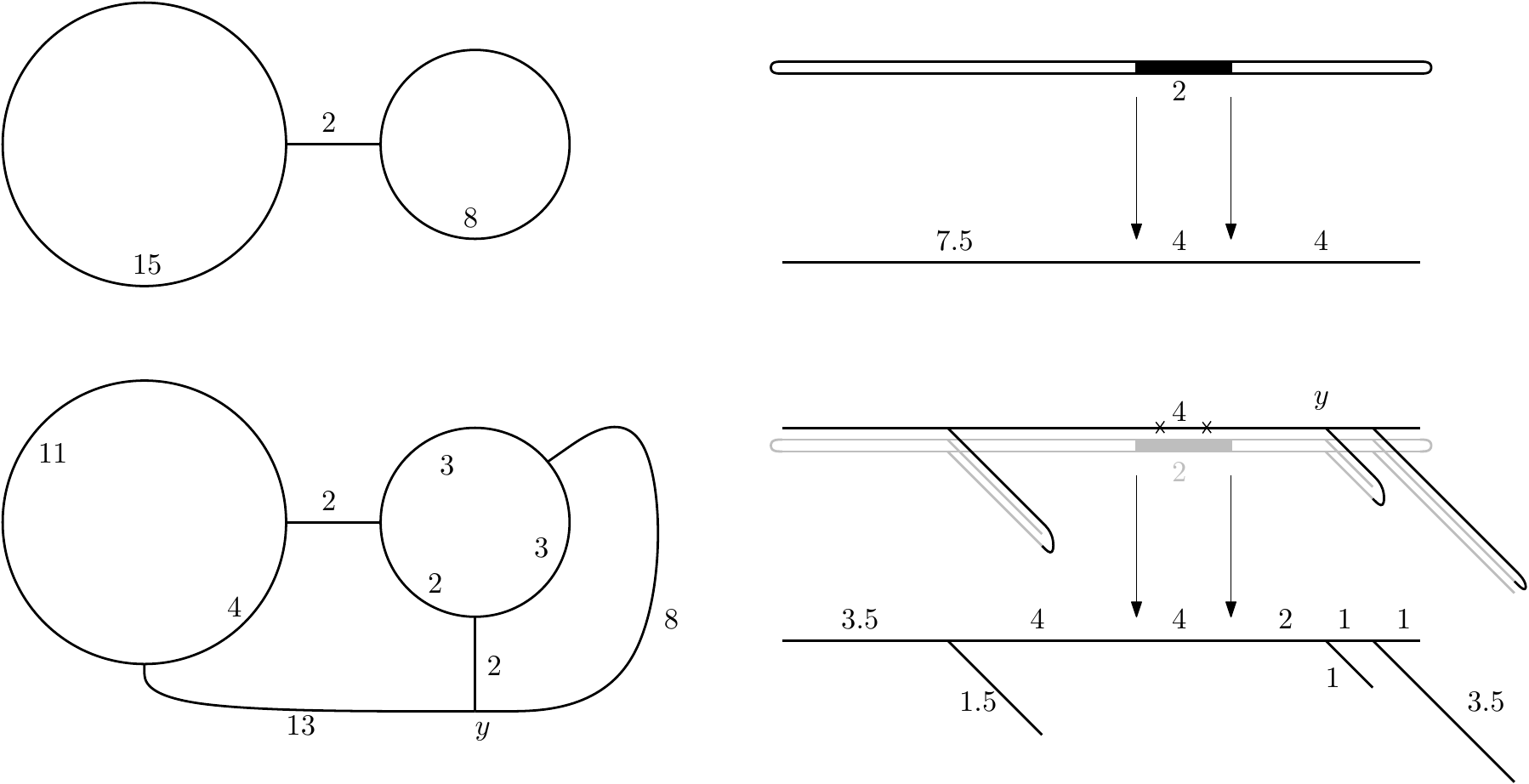}
\end{center}
Focussing on the latter morphism $\phi$, note that the segment between
the two arrows has length $4$ in the tree, and also in one of the copies
of the tree above, but length $2$ where the other two copies are glued
together. The two marked points are in the graph's integral set, but have
a strictly half-integral point in their fibre. So they would {\em not}
be a valid choice for $v_0$ in condition (3) above.
\end{ex}

\subsection*{Cactus graphs}

A metric graph $\Gamma$ is called a {\em cactus graph} if any two simple
cycles (i.e., injective, continuous images of $S^1$) intersect in at
most one point.

\begin{prop}
Any metric cactus graph $\Gamma$ has a modification $\Gamma'$ with a
tropical morphism $\phi$ from $\Gamma'$ to a tree $T$, of degree $\lceil
\frac{g(\Gamma)+2}{2} \rceil $, with the following additional constraints:
\begin{enumerate}
\item if $g(\Gamma)$ is odd and $v_1 \in \Gamma$ is any point, then $\phi$
can be chosen such that $m_\phi(v_1) = 2$ and that moreover $k \geq 2l-1$
where $k$ is the valency of $v_1$ in $\Gamma$ and $l$ is the valency of
$\phi(v_1)$ in $T$;
\item if $\Gamma$ has an integral set $S$ (containing $v_1$
if $g(\Gamma)$ is odd), then $\Gamma'$ has an integral set $S'$
containing $S$, $\phi(S')$ is contained in an integral set of $T$,
and $\phi^{-1}(\phi(S')) \subseteq S'$.
\end{enumerate}
\end{prop}

The former condition implies that the Riemann-Hurwitz inequality $k-2 \geq
2(l-2)$ holds at $v_1$ with a strict inequality.  The latter condition
is stronger than condition (3) in the proof above, where we require only
that the fibre through {\em some} integral point intersects $\Gamma$
only inside $S$. Example~\ref{ex:final} below shows why we could not impose
this stronger condition earlier.

\begin{proof}
We proceed by induction on $g$. For $g=0$ we take $T=\Gamma'=\Gamma$
and $\phi$ the identity map. For $g=1$ let $C$ be the unique simple cycle
in $\Gamma$, of length $a>0$, and let $w$ be the point on $C$
closest to the prescribed point $v_1$. The $2:1$ map from $C$ to an edge
with branchpoints $w$ and the point at distance $a/2$ from $w$ extends
to a modification of $\Gamma$ that has slope $1$ everywhere
except for slope $2$ on the segment connecting $v_1$ and $w$:\\
\begin{center}
  \includegraphics[scale=.78]{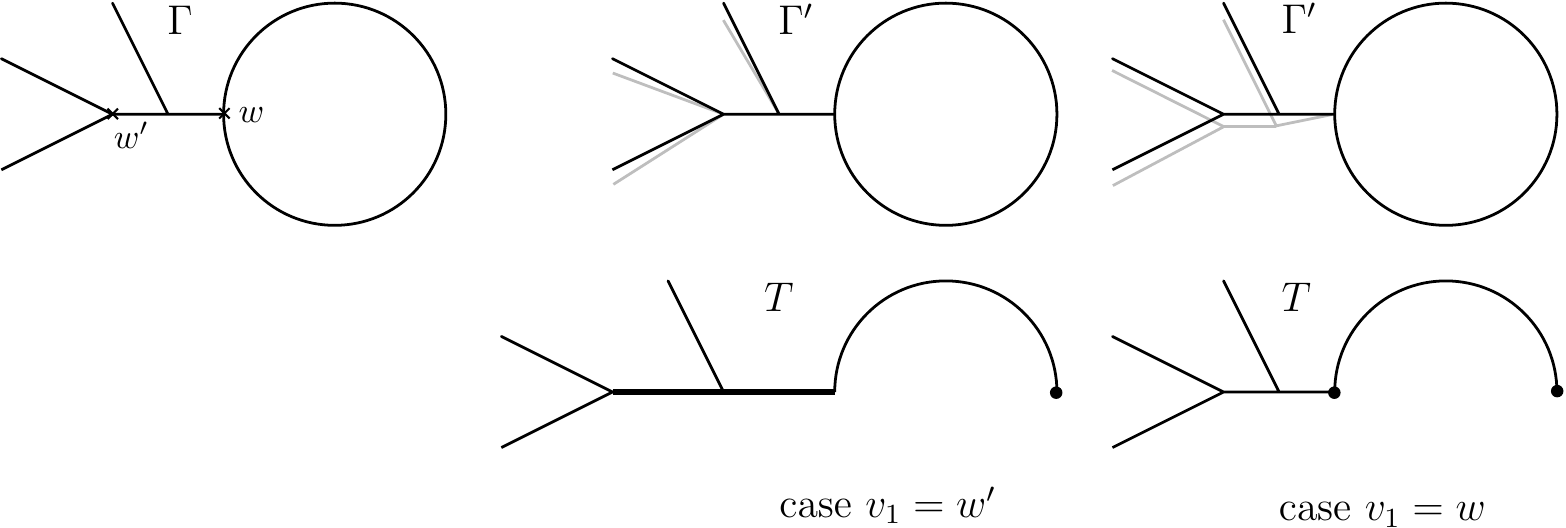}
\end{center}
For (1) we note that $k$ is at least $2l-1$, as required. For the integrality
condition (2), we note that any integral set $S$ of $\Gamma$ has a
unique extension to an integral set $S'$ of $\Gamma'$, and that there is
a unique integral set in $T$ containing the image of $S'$. Note that the
latter inclusion is strict if the segment from $v_1$ to $w$ has positive
(integral) length. Yet, $\phi^{-1}(\phi(S))=S$, as required.

If $\Gamma$ has higher genus, then we can write it as $\Gamma=\Gamma_1
\cup \Gamma_2$ where $\Gamma_1,\Gamma_2 \subseteq \Gamma$ are cactus
graphs of lower genus than $\Gamma$ that intersect in a single point $y$
of $\Gamma$, which we can chose integral if $\Gamma$ has an integral
set. Write $g:=g(\Gamma)$ and $g_i:=g(\Gamma_i)$ for $i=1,2$, and note
that $g=g_1+g_2$. Moreover, if $g$ is odd, then we can (and do) choose the
decomposition such that $g_1$ is odd and that the prescribed point $v_1$
lies in $\Gamma_1$. Furthermore, if $\Gamma$ has an integral set $S$,
then $S \cap \Gamma_i$ is an integral set for $\Gamma_i$ for each $i=1,2$.

By the induction hypothesis, there are modifications $\Gamma_i',\ i=1,2$
of the $\Gamma_i$ with tropical morphisms $\phi_i$ of degrees $\lceil
g_i/2 \rceil + 1$ to trees $T_i$ which further satisfy conditions (1)
and (2) in the proposition. Here we choose $v_1$ equal to $y$ for both
$\phi_1$ and $\phi_2$ if both $g_1$ and $g_2$ are odd (in this case,
since $g$ is even, no point $v_1$ had yet been prescribed).

Let $T$ be the tree obtained by gluing $T_1$ and $T_2$ at $\phi_1(y)$ and
$\phi_2(y)$, respectively, and let $\Gamma'$ be the metric graph obtained
by gluing $\Gamma_1',\Gamma_2'$ at $y$. The metric graph $\Gamma'$
is a modification of $\Gamma$, and an integral set of $\Gamma$ extends
uniquely to one of $\Gamma'$. Let $\psi:\Gamma' \to T$ be the map restricting
to $\phi_i$ on $\Gamma_i'$. Then $\psi$ is harmonic except in the
points in $Y:=\psi^{-1}(\psi(y))$. Apart from $y$, which belongs to
both $\Gamma_i'$, the points in $Y$ are either in $\Gamma_1'$ or in
$\Gamma_2'$ but not both. Let $v_0:=y,v_1,\ldots,v_a$ be the points in
$Y \cap \Gamma_1'$ and let $w_0:=y,w_1,\ldots,w_b$ be the points in $Y
\cap \Gamma_2'$, and write $m_i:=m_{\phi_1}(v_i)$ for $i=0,\ldots,a$
and $n_j:=m_{\phi_2}(w_j)$ for $j=0,\ldots,b$.  We modify $\Gamma'$ by
grafting $m_i$ copies of $T_2$ at $v_i$ for $i=1,\ldots,a$ (not yet at
$v_0=y$!) and grafting $n_j$ copies of $T_1$ at $w_j$ for $j=1,\ldots,b$
(not yet at $w_0=y$!), and we extend $\psi$ to these copies by their
natural maps into $T$. This renders $\psi$ harmonic at $v_1,\ldots,v_a$
and $w_1,\ldots,w_b$, and moreover restores the Riemann-Hurwitz condition
there---e.g., to the valency $v_i$ one adds $m_i$ times the valency of
$\phi_2(y)$ in $T_2$, which is exactly $m_i$ times what was added to
the valency of $\phi_1(v_i)=\phi(y)$ by attaching $T_2$.

So we need only establish harmonicity and the Riemann-Hurwitz condition
at $y$. Let $d_i=\lceil g_i/2 \rceil +1$ be the degree of $\phi_i$.
First, assume that $g_2$ is even. Then the required degree of the map
$\phi$ equals
\[ d:=\lceil g/2 \rceil + 1 = \lceil g_1/2 \rceil + g_2/2 + 1 = d_1+d_2-1. \]
We graft $m_0-1$ further copies of $T_2$ and $n_0-1$ copies of $T_1$
to $y$; this yields the final modification $\Gamma''$ of $\Gamma$ with
the natural extension $\phi:\Gamma'' \to T$ of 
$\psi$. This extension is harmonic everywhere by construction. To check
the Riemann-Hurwitz condition at $y$, let $k_i,k$ denote the valency of
$y$ in $\Gamma_i'$ and $\Gamma''$, respectively, and let $l_i,l$ denote
the valency of $\phi_i(y)$ and $\phi(y)$ in $T_i,T$, respectively.
Then we have $l=l_1+l_2$ and $k=k_1+k_2+(m_0-1)l_2+(n_0-1)l_1$, so
that
\begin{align*}
k-2 &= k_1+k_2+(m_0-1)l_2+(n_0-1)l_1 -2 \\
&\geq m_0(l_1-2)+2+n_0(l_2-2)+2+(m_0-1)l_2+(n_0-1)l_1 -2 \\
&=(m_0+n_0-1)(l_1+l_2-2)\\
&=m_{\phi}(y)(l-2),
\end{align*}
where the inequality follows from the Riemann-Hurwitz inequalities for
the $\phi_i$. 

Second, assume that $g_2$ is odd; then, by assumption, so is $g_1$.
With notation as above we now have
\[ d=\lceil g/2 \rceil + 1 = d_1+d_2-2. \]
Moreover, since we had chosen $y$ as the prescribed point for both
$\phi_1$ and $\phi_2$, we have $m_0=n_0=2$ by property (1). This means
that we need not graft further trees at $y$ and the map $\phi:\Gamma''
\to T$ constructed so far is harmonic there. To check the
Riemann-Hurwitz conditions, we compute 
\[ k-2=k_1+k_2-2 \geq 2l_1-1+2l_2-1-2=2(l-2)=m_\phi(y)(l-2), \]
where we have used that $k_i \geq 2l_i-1$ by property (1).

Finally, if $\Gamma$ has an integral set $S$, then $\Gamma''$
has a unique integral set $S'$ containing $S$, and it contains the
integral sets $S'_i=S' \cap \Gamma_i'$ as well as suitable integral
sets of the trees grafted onto $\Gamma'$ to arrive at $\Gamma''$. The
points in $\phi^{-1}(\phi(S'))$ that are not in the union of the sets
$\phi_i^{-1}(\phi_i(S'_i))$ are in those integral sets of the grafted 
trees, hence also in $S'$. 
\end{proof}

\begin{re}
\begin{enumerate}
\item 
Combining this subsection with the previous one, we find that metric
graphs in the open cone $C_G$ have a modification with a tropical morphism
to a tree that only has slopes $1$ and $2$.
\item 
The existence of divisors of {\em higher} rank on cactus graphs,
under the condition that the Brill-Noether number is nonnegative,
was studied by Jorn van der Pol in his Bachelor's thesis \cite{Pol11}
under an additional assumption on the cactus graph.
\item 
The existence of integral rank-one divisors of degree $\lceil (g+2)/2
\rceil$ on an arbitrary graph of genus $g$ remains conjectural.
Conceivably, Backman's approach to linear equivalence using graph
orientations \cite{Backman14} could lead to such a result. In any case,
following our approach in this paper, we do not see how to cross the
boundary of the cone $C_G$.

\end{enumerate}
\end{re}


\section*{Acknowledgments}

We thank Aart Blokhuis, who provided the proof of
Lemma~\ref{lm:Partition} reproduced in this paper. JD was partially supported
by Vidi and Vici grants from the Netherlands Organisation for Scientific Research
(NWO).

\bibliographystyle{alpha}
\bibliography{diffeq,draismajournal}

\end{document}